\documentclass[reqno]{amsart}
\usepackage[marginratio=1:1]{geometry}
\usepackage[hidelinks]{hyperref}
\usepackage{calc}
\newsavebox\CBox
\newcommand\hcancel[2][0.5pt]{%
  \ifmmode\sbox\CBox{$#2$}\else\sbox\CBox{#2}\fi%
  \makebox[0pt][l]{\usebox\CBox}%
  \rule[0.5\ht\CBox-#1/2]{\wd\CBox}{#1}}
\usepackage{blindtext}
\usepackage{color}
\usepackage{hyperref,url}
\usepackage{comment}
\usepackage{float}
\usepackage{hyperref}
\usepackage{enumerate}
\usepackage{enumitem}   
\usepackage{bbm}
\usepackage{cancel}
\usepackage{mathrsfs}
\usepackage{colonequals}
\usepackage{pdfpages}

\usepackage{amsmath,amsfonts,amsthm,bm}
\usepackage{mathrsfs}  
\usepackage{xcolor}
\usepackage{amsfonts}
\usepackage{amssymb}
\usepackage{amsmath}
\usepackage{mathtools}
\usepackage{mathrsfs}  

\usepackage[normalem]{ulem}

\numberwithin{equation}{section}
\usepackage[toc,page]{appendix}
\usepackage{tikz-cd}
\theoremstyle{definition}

\newtheorem{definicao}{Definition}[section]

\theoremstyle{plain}

\newtheorem{teorema}[definicao]{Theorem}
\newtheorem{suposicao}{Hypothesis}

\newtheorem{lema}[definicao]{Lemma}

\newtheorem{corolario}[definicao]{Corollary}

\newenvironment{TheoremproofA}[1]{\par\noindent{\emph{Proof of Corollary \ref{cor1}}} \space#1}{\leavevmode\unskip\penalty9999 \hbox{}\nobreak\hfill\quad\hbox{$\qed$}}
\newenvironment{TheoremproofB}[1]{\par\noindent{\emph{Proof of Theorem  \ref{thm1}.}}  }{\leavevmode\unskip\penalty9999 \newline\hbox{}\nobreak\hfill \quad\hbox{$\qed$}}
\newenvironment{TheoremproofC}[1]{\par\noindent{\textit{Proof of Theorem \ref{thm2}}} \space#1}{\leavevmode\unskip\penalty9999 \hbox{}\nobreak\hfill\quad\hbox{$\qed$}}
\usepackage{scalerel}

\newenvironment{TheoremproofE}[1]{\par\noindent{\emph{Proof of Theorem \ref{thm3}}} \space#1}{\leavevmode\unskip\penalty9999 \hbox{}\nobreak\hfill\quad\hbox{$\qed$}}
\definecolor{roxo}{rgb}{0.44, 0.16, 0.39}
\definecolor{ao(english)}{rgb}{0.0, 0.5, 0.0}
\definecolor{dmagenta}{RGB}{139, 0, 139}
\definecolor{dgreen}{RGB}{0,90,0}
\definecolor{navy}{RGB}{0,0,128}

\usepackage{stackengine}

\def\d{\mathrm d}
\def \d {\mathrm{d}}

\definecolor{iblue}{RGB}{0, 35, 194}

\title[Propagation of chaos]
{Uniform in time propagation of chaos for noisy mean-field coupled maps}

\author{Giuseppe Tenaglia}
\address{Department of Mathematics, Imperial College London,
London SW7 2AZ, UK}
\email{giuseppe.tenaglia20@imperial.ac.uk}

\author{Matteo Tanzi}
\address{Department of Mathematics, King's College London,
Strand, London WC2R 2LS, UK}
\email{matteo.tanzi@kcl.ac.uk}



\begin{document}

\begin{abstract}
We study discrete-time $N$-dimensional mean-field systems on the torus subject to additive noise whose probability density is bounded away from zero. Given a Lipschitz one-particle map and interaction term, we prove that, if the lower bound on the noise density is sufficiently large, the $N$-dimensional transfer operator $\mathcal P_N$ preserves a dimension-independent class of sub-Gaussian probability measures. Moreover, $\mathcal P_N$ is a one-step contraction in the Dobrushin-Wasserstein distance and therefore admits a unique invariant measure $\rho_N$, which is itself sub-Gaussian. Under the same condition on the noise strength, we show that the associated self-consistent transfer operator is a one-step contraction in Wasserstein distance and hence admits a unique fixed point $\rho$. We further prove, using the contraction estimates established in Theorem \ref{thm1}, that $\mathcal P_N$ preserves an $O(N^{-1/2})$ neighbourhood of the product measure $\rho^{\otimes N}$ in the Dobrushin-Wasserstein metric, and in particular that $\rho_N$ belongs to this neighbourhood. Finally, we establish uniform-in-time propagation of chaos in the Dobrushin-Wasserstein metric and, under the additional assumption that the noise density is of bounded variation, in total variation for every fixed-dimensional marginal. 
\end{abstract}

\maketitle

Many systems of scientific interest can be described as collections of interacting units, whose mutual interactions determine the global behaviour of the system. Natural examples arise in population dynamics, neuroscience, statistical physics, epidemiology, and deep learning \cite{ChaintronDiez2022,Touboul2012,SirignanoSpiliopoulos2020}. From a mathematical point of view, such systems pose a serious challenge: even when the dynamics of each individual unit is relatively simple, the interactions among many of them can generate remarkably complex collective behavior.

A particularly challenging class is given by discrete-time mean-field coupled systems. They consist of many interacting units, whose pairwise interaction scales inversely with the size of the system. When the number of units is very large, the evolution of the empirical state is well approximated by a nonlinear operator---called the self-consistent transfer operator---acting on the space of probability measures associated with the dynamics of a single element \cite{TanziReview2023,Galatolo2022}.

There is by now a substantial literature on self-consistent transfer operators associated to mean-field coupled systems \cite{TanziReview2023,Galatolo2022,BahsounLiveraniSelley2023,BahsounKorepanov2024,CastorriniGalatoloTanzi2025,CastorriniGalatoloTanzi2026}. Much less is known, however, about the ergodic properties of the finite-dimensional dynamics. Indeed, many of the results obtained so far bypass a direct analysis of the high-dimensional systems and focus on comparing the evolution of the empirical measure of the particle system with the self-consistent dynamics \cite{DelMoralRio2011,BudhirajaMajumder2015,Vuckovic2026}. These results aim to establish propagation of chaos \cite{Sznitman1991,ChaintronDiez2022}: starting from an asymptotically independent family of initial distributions, the $k$-dimensional marginals of the $N$-particle system converge, as $N\to\infty$, to the $k$-fold product of the corresponding self-consistent one-particle evolution.

Currently, there is no well-developed ergodic-theoretic framework for the study of high-dimensional dynamics. At a fundamental level, it is not clear which geometric and statistical quantities continue to provide a meaningful description of the collective behavior as the number of components grows. The main obstruction is the curse of dimensionality: many quantities that effectively describe low-dimensional dynamics deteriorate with the dimension itself.

A first step toward addressing this gap was taken in \cite{Tanzi2022}, which introduced an operator-theoretic framework for mean-field coupled systems driven by uniformly expanding maps. Provided expansion dominates interaction, the dynamics preserve a class of probability measures that remain quantitatively close to product measures, with spatial dependencies controlled via a Dobrushin-type matrix. Measures in this class consequently satisfy concentration inequalities typical of weakly dependent systems. While \cite{Tanzi2022} established that the invariant measure belongs to this class—exploiting this to prove propagation of chaos—their approach does not yield a dimension-uniform rate of convergence to equilibrium for the finite-dimensional system.

From an ergodic point of view, such a quasi-product structure is important because it gives a finite-dimensional quantitative formulation of the asymptotic independence underlying propagation of chaos: as $N$ grows, the dependence between individual particles becomes weaker and the system increasingly resembles a product system.

Motivated by this perspective, in the present work we seek analogous ergodic properties for mean-field systems in which the regularizing role played by expansion in \cite{Tanzi2022} is instead provided by fully supported additive noise. More specifically, we consider a class of $N$-dimensional random mean-field systems on the $N$-torus $\mathbb{T}^N$ of the form
\begin{align*}
x_i^{n+1}
=
f(x_i^n)
+
\frac{1}{N}\sum_{j=1}^N h(x_i^n,x_j^n)
+
\omega_i^n
\mod 1,
\qquad i=1,\ldots,N,
\end{align*}
where $f$ and $h$ are Lipschitz and the random variables $\omega_i^n$ are i.i.d., with common probability density $p$ on $[0,1]$ bounded away from zero.

Our first main result is Theorem \ref{thm1}, in which we show that, provided the lower bound on the noise density $p$ is sufficiently large relative to the Lipschitz constants of $f$ and $h$, the transfer operator associated with the full $N$-dimensional dynamics preserves a dimension-independent class of sub-Gaussian probability measures and is a one-step contraction, with contraction rate independent of $N$, on $\mathcal{M}_1(\mathbb{T}^N)$ endowed with the $N$-dimensional Dobrushin-Wasserstein distance. In particular, the $N$-dimensional system admits a unique sub-Gaussian invariant probability measure.

Under the same condition on the noise strength, Theorem \ref{thm2} establishes that the associated self-consistent transfer operator admits a unique fixed point $\rho$. Furthermore, it shows that the $N$-dimensional dynamics preserves an $O(N^{-1/2})$ neighbourhood of $\rho^{\otimes N}$ in the Dobrushin-Wasserstein metric. In particular, its unique invariant measure belongs to this neighbourhood.

The contraction estimate obtained in Theorem \ref{thm1} is also used in Theorem \ref{thm3}. There, we compare, in the Dobrushin-Wasserstein metric, the evolution of an exchangeable $N$-dimensional measure $\mu$ with the evolution of $N$ independent copies of the self-consistent transfer operator acting on the one-dimensional marginal of $\mu$. More precisely, we show that the effect of the initial dependence between the coordinates of $\mu$ decays exponentially fast, up to an $O(N^{-1/2})$ error. If, in addition, the noise density is of bounded variation, we obtain an analogous comparison in total variation for every fixed-dimensional marginal. Corollary \ref{cor1} then exploits these estimates to establish uniform-in-time propagation of chaos in both metrics.

To the best of our knowledge, this is the first work to exploit the regularizing effect of fully supported additive noise to construct an invariant class of sub-Gaussian probability measures whose concentration properties remain uniform in the dimension. In this sense, the sub-Gaussian class introduced here plays a role analogous to the quasi-product classes constructed in \cite{Tanzi2022} for uniformly expanding coupled maps, providing a quantitative notion of product-like structure for the full finite-dimensional system. In the present random setting, the noise yields in addition a dimension-uniform contraction of the full $N$-dimensional transfer operator. The choice of the Dobrushin-Wasserstein distance is essential for obtaining such a result: a direct Doeblin argument in total variation \cite{MeynTweedie2009} would lead to a contraction coefficient  which converges exponentially fast to $1$ as $N \to \infty.$

The dimension-independent contraction rate established in Theorem \ref{thm1} is the key ingredient in Theorem \ref{thm2}, where we prove that the invariant measure $\rho_N$ of the $N$-dimensional system is $O(N^{-1/2})$ close to the $N$-fold product $\rho^{\otimes N}$ of the equilibrium of the self-consistent evolution. Together with the sub-Gaussianity of $\rho_N$, this provides a quantitative finite-dimensional description of the system at equilibrium: for large $N$, its stationary state resembles that of $N$ independent particles distributed according to $\rho$. Theorem \ref{thm3} provides the finite-time counterpart of this equilibrium picture. It shows that, starting from an exchangeable distribution with arbitrary dependence among the particles, the dynamics forgets these dependencies exponentially fast, up to an $O(N^{-1/2})$ error. Together, these results provide a quantitative description of how the system behaves at each fixed dimension $N$, both throughout its evolution and at equilibrium, without having to pass to the limit $N\to\infty$.

The proofs of Theorems \ref{thm2} and \ref{thm3} follow a general scheme that also appears in the approaches of Del Moral and Rio \cite{DelMoralRio2011} and, more recently, Vuckovic \cite{Vuckovic2026}. In those works, the authors combine a local contraction property of the self-consistent dynamics with a certain conditional independence of the particles to compare the evolution of the system's state with the limiting nonlinear dynamics. By contrast, we combine the contraction of the full N-dimensional transfer operator with concentration estimates for the product measures generated by the self-consistent evolution to estimate the error between the two dynamics in the Dobrushin-Wasserstein metric.

With their method, \cite{DelMoralRio2011,Vuckovic2026} obtain uniform-in-time propagation of chaos at the $O(N^{-1/2})$ scale for product initial data, in a Wasserstein framework essentially equivalent to the Dobrushin-Wasserstein formulation considered here. Our approach recovers this estimate and extends it to arbitrary asymptotically independent exchangeable families by explicitly tracking the initial dependence. Furthermore, when the noise density is of bounded variation, we upgrade the result to total variation for every fixed-dimensional marginal (see Corollary \ref{cor1}).

This work is organized as follows. In Section \ref{sec2} we state our Hypothesis and our main results. In Section \ref{Sec3} we prove Theorem \ref{thm1}, in Section \ref{sec4} we prove Theorem \ref{thm2} and in Section \ref{sec5} we prove Theorem \ref{thm3} and Corollary \ref{cor1}.

\section{Hypothesis and main results}\label{sec2}
Let $\mathbb{T}$ denote the one-dimensional torus $\mathbb{R}/\mathbb{Z}$, endowed with the arc-length distance $d(\cdot,\cdot).$ For $N \in \mathbb{N}$, let $\mathbb{T}^N$ denote the $N$-torus. We consider $N$-dimensional Markov chains on $\mathbb{T}^N$ satisfying the following assumption.
\begin{suposicao}\label{A}
There exist a one-particle map  $f \in \operatorname{Lip} (\mathbb{T},\mathbb{R})$ and a  coupling function $h \in \operatorname{Lip}(\mathbb{T}^2,\mathbb{R}),$ such that the $N$-dimensional chain has the form
\begin{equation}\label{Markovchain}
x_i^{n+1} = f(x_i^n) +\frac{1}{N}\sum_{j=1}^N h(x^n_i,x^n_j)+\omega_i^n \mod1,  \qquad i = 1,\dots N,
\end{equation}
where, for all $n\ge 0$ and $i=1,\dots,N$, the random variables $\omega_i^n $ are i.i.d. and sampled from $[0,1]$ with a probability density $p,$ which is
extended periodically to $\mathbb R$ and satisfies
\begin{equation}\label{doeblin}
p(x) >c>0\qquad \forall x\in [0,1].
\end{equation}
Furthermore, we assume
\begin{equation}\label{kappa}
\kappa:= (1-c)\left(\operatorname{Lip}(f)+\operatorname{Lip}_1(h)+\operatorname{Lip}_2(h)\right)<\frac{\sqrt{3}}{2},
\end{equation}
where $\operatorname{Lip}(f)$ denotes the Lipschitz constant of $f$
$$
\operatorname{Lip}_1(h) :=\sup_{u_1\neq u'_1,u_2 \in \mathbb{T} } \frac{|h(u_1,u_2)-h(u_1',u_2)|}{d(u_1,u_1')},
$$
and $\operatorname{Lip}_2(h)$ is defined analogously.
\end{suposicao}
Given a Markov chain satisfying \eqref{Markovchain}, its associated transfer operator reads
\begin{equation}\label{transferop}
\mathcal{P}_N(\mu)(dx) = \int \prod_{j=1}^N p\left(x_j-f(u_j)-\frac{1}{N}\sum_{k=1}^Nh(u_j,u_k)\right)dx_jd\mu(u), \qquad \forall \mu \in \mathcal{M}_1(\mathbb T^N),
\end{equation}
where $\mathcal{M}_1(\mathbb T^N)$ denotes the space of probability measures on $\mathbb{T}^N$. 
In this work, we endow $\mathcal{M}_1(\mathbb{T}^N)$ with the $N$--dimensional Dobrushin-Wasserstein distance \begin{equation}\label{wasserstainspace}
W^1_N(\mu_1,\mu_2) = \sup_{g \colon \mathbb{T}^N \to \mathbb{R} \colon \operatorname{Lip}^N(g) \le 1}\int g(x) (\mu_1-\mu_2)(dx) \qquad \forall \mu_1,\mu_2 \in \mathcal{M}_1(\mathbb{T}^N),
\end{equation}    
where 
\begin{equation}\label{lipN}
\operatorname{Lip}^N(g):=\sum_{i=1}^N \operatorname{Lip}_i(g), 
\end{equation}
and, for any observable $g \colon \mathbb{T}^N \to \mathbb{R}$
\begin{equation}\label{lipi}
\operatorname{Lip}_i(g):= \sup_{x_{-i} \in \mathbb T^{N-1}} \sup_{x_i\neq x'_i}\frac{|g(x_i,x_{-i})-g(x'_i,x_{-i})|}{d(x_i,x'_i)}.
\end{equation}
Here, $x_{-i}$ denotes the $(N-1)$-dimensional vector obtained by removing the $i$-th coordinate
\begin{equation}\label{xminus}
x_{-i} = (x_v)_{v \in \{1,\dots, N\}\setminus \{i\}}
\end{equation}

Given a probability measure $\mu \in \mathcal{M}_1(\mathbb{T}^N),$ and an observable $g \colon \mathbb{T}^N \to \mathbb{R}$, we denote by $\mathbb{E}_{\mu}[g]$, the expectation of $g$ with respect to $\mu$.

\begin{definicao}\label{Subgaussian}
Let $\mu$ be a probability measure on $\mathbb{T}^N$. We say that $\mu$ is sub-Gaussian if, for any observable  $g \colon \mathbb{T}^N \to \mathbb{R}$
\begin{align*}
\mathbb{E}_{\mu}\left[e^{\lambda (g-\mathbb{E}_\mu[g])}\right] \le \exp\left(\frac{\lambda^2}{2}\sum_{i=1}^N \operatorname{Lip}_i(g)^2\right), \qquad \forall \lambda \in \mathbb{R}.
\end{align*}
\end{definicao}

It follows by Chernoff's bound \cite[Sections 2.2-2.3]{boucheron2013concentration} that if $\mu$ is sub-Gaussian, then it satisfies a concentration inequality of the form
\begin{equation}\label{chernoffbound}
\mu \left\{|g-\mathbb{E}_\mu[g] | > t\right\} \le 2 \exp\left(-\frac{t^2}{2\sum_{i=1}^N \operatorname{Lip}_i(g)^2} \right).
\end{equation}

\begin{teorema}\label{thm1}
For $N\ge 2$, let $\mathcal{P}_N$ be the transfer operator associated to a Markov chain satisfying Hypothesis \ref{A}. Then, if $\mu \in \mathcal{M}_1(\mathbb{T}^N)$ is sub-Gaussian, so is $\mathcal{P}_N(\mu).$ Furthermore, let $\kappa$ be as in \eqref{kappa}. Then, for all $\mu_1, \mu_2 \in \mathcal{M}_1(\mathbb{T}^N)$
\begin{equation}\label{tensorizedcontracton}
W^1_N(\mathcal{P}_N(\mu_1),\mathcal{P}_N(\mu_2))\le \kappa W^1_N(\mu_1,\mu_2).
\end{equation}
In particular, $\mathcal{P}_N$ admits a unique fixed point $\rho_N.$ Furthermore, $\rho_N$ is sub-Gaussian.
\end{teorema}

When $N \to \infty,$ the limiting dynamics governing the state of the Markov chain in \eqref{Markovchain}  is described by the so-called  \emph{self-consistent transfer operator}
\begin{equation}\label{STO}
\mathcal{P}(\nu)(dx) = \int  p\left(x-f(u)-\int h(u,v)d\nu(v)\right)dxd\nu(u), \qquad \forall \nu \in \mathcal{M}_1(\mathbb T).
\end{equation}

Let $q \in \mathcal{M}_1(\mathbb{T})$ be a one-dimensional probability measure, and, for $C>0$ and $N\in \mathbb{N}$, let $B_{C,N}(q)$ denote  the ball of size $CN^{-\frac{1}{2}}$ in the Dobrushin-Wasserstein distance around the product measure $q^{\otimes N}:$
\begin{equation}\label{BCN}
B_{C,N}(q):=\left\{\mu \in \mathcal{M}_1(\mathbb{T}^N) \colon W^1_N(\mu,q^{\otimes N})\le CN^{-\frac{1}{2}}\right\}.
\end{equation}

\begin{teorema}\label{thm2}
Let $\mathcal{P}_N$ be the transfer operator associated to a Markov chain satisfying Hypothesis \ref{A}, with one-particle map $f \in \operatorname{Lip}(\mathbb{T}), h \in \operatorname{Lip}(\mathbb{T}^2),$ density $p$ satisfying \eqref{doeblin} and $\kappa$ as in \eqref{kappa}. Then, the following statements hold:
\begin{enumerate}
\item \label{one} the associated self-consistent transfer operator $\mathcal{P}$ in \eqref{STO} is a one-step contraction in $W^1_1$. In particular, for all $\mu_1,\mu_2 \in \mathcal{M}_1(\mathbb{T})$
\begin{equation}\label{stocontr}
W^1_1(\mathcal{P}(\mu_1),\mathcal{P}(\mu_2))\le \kappa W^1_1(\mu_1,\mu_2).
\end{equation}
In particular $\mathcal{P}$ admits a unique fixed point $\rho$.
\item \label{two} Let 
\begin{equation}\label{barC}
\bar C:= \frac{(1-c)\operatorname{Lip}_2(h)}{1-\kappa}.
\end{equation}
Then,  for all $N\ge 2$
\begin{equation}\label{firstinclusion}
\mathcal{P}_N B_{\bar  C,N}(\rho)\subset B_{\bar  C,N}(\rho),
\end{equation}
In particular, for all $N \ge 2$, the unique  invariant sub-Gaussian   measure $\rho_N$  of  $\mathcal{P}_N$ belongs to $B_{\bar C,N}(\rho)$ and satisfies
\begin{equation}\label{allI}
W^1_N(\rho_N,\rho^{\otimes N})
\le \bar C N^{-\frac12}.
\end{equation}
\end{enumerate}
\end{teorema}

We now recall the notion of an exchangeable probability measure.

\begin{definicao}\label{exchangeable}
We say that a probability measure $\mu \in \mathcal{M}^1(\mathbb{T}^N)$ is \emph{exchangeable} if, for every permutation $\sigma$ of $\{1,\dots,N\}$ and any bounded measurable $g \colon \mathbb{T}^N \to \mathbb{R}$, we have
\[
\int_{\mathbb{T}^N} g(x_1,\dots,x_N)\,d\mu
=
\int_{\mathbb{T}^N} g(x_{\sigma(1)},\dots,x_{\sigma(N)})\,d\mu.
\]
\end{definicao}

An immediate consequence of exchangeability is that the $k$-dimensional marginals of an $N$-dimensional exchangeable measure depend only on $k$, and not on the particular choice of coordinates onto which one projects. Moreover, since $\mathcal{P}_N$ commutes with coordinate permutations, it preserves exchangeability.

In our final main result, we establish a quantitative loss of dependence for exchangeable measures in the $W_N^1$ metric under Hypothesis \ref{A}. If, in addition, the noise density belongs to $\operatorname{BV}(\mathbb T)$, we obtain an analogous estimate in total variation. As a consequence, we deduce uniform-in-time propagation of chaos.

\begin{teorema}\label{thm3}
For $N \ge 2,$ let $\mu \in \mathcal{M}_1(\mathbb{T}^N)$ be exchangeable. Then, under the assumptions of Theorems  \ref{thm1} and \ref{thm2}, there exists $\bar C_0>0$, independent of $N,$ such that, for all $n\ge0$  
\begin{equation}\label{firstlasota}
W_N^1\left(
\mathcal P_N^n(\mu),
\left(\mathcal P^n(\pi_1\mu)\right)^{\otimes N}\right) \le \kappa^n W^1_N\left(\mu,\left(\pi_1\mu\right)^{\otimes N}\right) + \bar C_0N^{-\frac{1}{2}},
\end{equation}
where $\kappa$ is as in \eqref{kappa}. If we further assume that   $p \in \operatorname{BV}(\mathbb T)$, we  have, for all $n\ge 1$ 
\begin{equation}\label{secondlasota}
\|\pi_{\{1,\dots,k\}}\mathcal{P}_N^n(\mu)-\left(\mathcal P^n(\pi_1\mu)\right)^{\otimes k}\|_{TV} \le \bar{C}_{0,k}  \kappa^{n-1} W^1_N\left(\mu,\left(\pi_1\mu\right)^{\otimes N}\right)+ \bar{C}_{1,k}N^{-\frac{1}{2}},\qquad \forall 1\le k\le N.
\end{equation}
for some $\bar C_{0,k},\bar C_{1,k} >0.$
\end{teorema}
A corollary of Theorem \ref{thm3} is the following uniform-in-time propagation of chaos result.
\begin{corolario}\label{cor1} 
Assume that $\{\mu_N\}_{N\ge 2}$ is a family of exchangeable measures with $\mu_N \in \mathcal{M}_1(\mathbb T^N)$ satisfying, for every fixed $k\in\mathbb N$,
\begin{equation}\label{kcong}
W_k^1\left(
\pi_{\{1,\ldots,k\}}\mu_N,
(\pi_1\mu_N)^{\otimes k}
\right)
\longrightarrow 0,
\qquad\text{as }N\to\infty.
\end{equation}
Then,  under the assumptions of Theorems  \ref{thm1} and \ref{thm2}, 
\begin{equation}\label{propagationw1n}
\sup_{n\geq 0}
W_N^1\left(
\mathcal P_N^n(\mu_N),
\left(\mathcal P^n(\pi_1\mu_N)\right)^{\otimes N}
\right) \to 0,\qquad N\to \infty.
\end{equation}
If we further assume that the density $p$ in \eqref{doeblin} is a function of bounded variation, then, for every fixed $k\in\mathbb{N}$,
\begin{equation}\label{whatsithis}
\sup_{n\ge 1}
\left\|
\pi_{\{1,\dots,k\}}P_N^n(\mu_N)
-
\bigl(P^n(\pi_1\mu_N)\bigr)^{\otimes k}
\right\|_{\mathrm{TV}}
\longrightarrow 0,
\qquad N\to\infty.
\end{equation}
\end{corolario}
\section{Proof of Theorem \ref{thm1}}\label{Sec3}
Let $N \ge 2$. For $u \in \mathbb{T}^N$, define the kernels 
\begin{equation}\label{nuj}
\nu_{j,u}(dx_j)=p\left(x_j-f(u_j)-\frac{1}{N} \sum_{k=1}^N h(u_j,u_k)\right)dx_j,
\end{equation}
and the product measure
\begin{equation}\label{nuu}
\nu_u(dx) = \prod_{j= 1}^N \nu_{j,u}(dx_j), 
\end{equation}   
All the assertions of Theorem \ref{thm1} are a corollary of the following Lemma.
\begin{lema}\label{keylema}
For $N \ge 2,$ let us consider an $N$-dimensional Markov chain satisfying Hypothesis \ref{A} with one-particle map $f,$ coupling function $h$  and transition density $p.$ Then, for any observable $g \colon \mathbb{T}^N\to \mathbb{R},$ if we define
\begin{equation}\label{effgiu}
F_g(u):= \int g(x)\nu_u(dx), \qquad \forall u \in \mathbb{T}^N,
\end{equation}
with $\nu_u$ as in \eqref{nuu}, then we have
\begin{equation}\label{contractionobservable}
\operatorname{Lip}^N(F_g) \le \kappa \operatorname{Lip}^N(g)
\end{equation}
and 
\begin{equation}\label{squarecontractionobservable}
\sum_{i=1}^N\operatorname{Lip}_i^2(F_g) \le \kappa^2 \sum_{i=1}^N\operatorname{Lip}_i^2(g),
\end{equation}
with $\kappa$ as in \eqref{kappa}.
\end{lema}

\begin{proof}
Fix $i\in\{1,\dots,N\}$ and let $u,v\in\mathbb{T}^N$ differ only in the $i$-th coordinate.  Using \eqref{effgiu}, we decompose 
\begin{equation}\label{fuminsfv}
F_g(u)-F_g(v)
=\sum_{j=1}^N \left(\int g\,d\mu_{j-1}-\int g\,d\mu_j\right),
\end{equation}
where $\mu_0=\nu_u,\,\, \mu_N=\nu_v,$
with $\nu_u,\nu_v$ as in \eqref{nuu}, and  for $j=1,\dots,N-1$ 
\begin{align*}
\mu_j=\prod_{k=1}^j \nu_{k,v}(dx_k)\prod_{k=j+1}^N \nu_{k,u}(dx_k),
\end{align*}
with $\nu_{k,u}$ as in \eqref{nuj}.

Fix $j\in\{1,\dots,N\}$. By Fubini's theorem 
\begin{equation}\label{gdecomp}
\int g\,d\mu_{j-1}-\int g\,d\mu_j
=
\int \left[\int g(x)\,\left(\nu_{j,u}(dx_j)-\nu_{j,v}(dx_j)\right)\right]\mu_{-j}(d x_{-j}),    
\end{equation}
where 
\begin{align*}
\mu_{-j}(d x_{-j}):=\prod_{k<j}\nu_{k,v}(dx_k)\prod_{k>j}\nu_{k,u}(dx_k),
\end{align*}
and $x_{-j}$ is as in \eqref{xminus} with $i =j.$ By \eqref{nuj}
\begin{align*}
\int g(x)\,\nu_{j,u}(dx_j)-\int g(x)\,\nu_{j,v}(dx_j) = \int g(x)\left[p\left(x_j-a_j(u)\right) -p\left(x_j-a_j(v)\right) \right ]dx_j,
\end{align*}
where 
\begin{equation}\label{ajei}
a_j(s) = f(s_j)+\frac{1}{N} \sum_{k=1}^N h(s_j,s_k),    \qquad s \in \mathbb{T}^N. 
\end{equation}
Let $c$ as in \eqref{doeblin} and   define
\begin{align*}
q(x) := \frac{p(x)-c}{1-c}.
\end{align*}
Then, $q$ is a probability density and, for all $x$, we can write 
\begin{equation}\label{central_Deco}
p(x) = c + (1-c)q(x).
\end{equation}
As a result, we have
\begin{align*}
&\int g(x_j,x_{-j})
\left[p(x_j-a_j(u))-p(x_j-a_j(v))\right]\,dx_j
\\
&\qquad =
(1-c)\int
\left[
g(x_j+a_j(u),x_{-j})
-
g(x_j+a_j(v),x_{-j})
\right]q(x_j)\,dx_j
\\
&\qquad \le
(1-c)\operatorname{Lip}_j(g)\,
|a_j(u)-a_j(v)|.
\end{align*}

The above inequality, combined with \eqref{fuminsfv}, \eqref{gdecomp} and \eqref{ajei}, gives 
\begin{equation}\label{estimateone}
|F_g(u)-F_g(v)|
\le (1-c)\sum_{j=1}^N \textnormal{Lip}_j(g)|a_j(u)-a_j(v)|.
\end{equation}

If $j\neq i$, then $u_j=v_j$, so only the dependence on the $i$-th coordinate inside the mean-field term changes:
\begin{equation}\label{est2}
|a_j(u)-a_j(v)|
\le \frac{1}{N}\textnormal{Lip}_2(h)\,d(u_i,v_i).
\end{equation}

If $j=i$, then both the local term and the interaction term vary:
\begin{equation}\label{Est3}
\begin{aligned}
|a_i(u)-a_i(v)|
&\le |f(u_i)-f(v_i)|
+\frac{1}{N}\sum_{k=1}^N |h(u_i,u_k)-h(v_i,v_k)|\\
&\le \Big(\textnormal{Lip}(f)+\textnormal{Lip}_1(h)\Big)d(u_i,v_i)
+\frac{1}{N}\textnormal{Lip}_2(h)\,d(u_i,v_i).    
\end{aligned}
\end{equation}

Substituting \eqref{est2},\eqref{Est3} into \eqref{estimateone}, we obtain 
\begin{equation}\label{Lipif}
\textnormal{Lip}_i(F_g)
\le (1-c)\Big(\textnormal{Lip}(f)+\textnormal{Lip}_1(h)\Big)\textnormal{Lip}_i(g)
+\frac{1-c}{N}\textnormal{Lip}_2(h)\sum_{j=1}^N \textnormal{Lip}_j(g).    
\end{equation}
Summing  over $i=1,\dots,N,$ we obtain \eqref{contractionobservable}.

By \eqref{Lipif},  we have
\begin{align*}
\sum_{i=1}^N \textnormal{Lip}_i(F_g)^2
\le A^2 \sum_{i=1}^N \textnormal{Lip}_i(g)^2
+ 2AB \left(\sum_{i=1}^N \textnormal{Lip}_i(g)\right)^2
+ N B^2 \left(\sum_{j=1}^N \textnormal{Lip}_j(g)\right)^2.
\end{align*}
where we introduced the notation
\begin{align*}
A := (1-c)\Big(\textnormal{Lip}(f)+\textnormal{Lip}_1(h)\Big),
\qquad
B := \frac{(1-c)}{N}\textnormal{Lip}_2(h).
\end{align*}
Squaring \eqref{Lipif}, summing the above inequalities over $i=1,\dots,N,$ and applying the Cauchy-Schwarz inequality we obtain \eqref{squarecontractionobservable}.
\end{proof}
\begin{TheoremproofB}

Let us fix $N\ge 2$, and the $N$-dimensional Markov chain satisfying Hypothesis \ref{A}, with one-particle map $f \in \operatorname{Lip}(\mathbb{T}), h \in \operatorname{Lip}(\mathbb{T}^2),$ density $p$ satisfying \eqref{doeblin} and $\kappa$ as in \eqref{kappa}.

\textbf{Step 1: sub-Gaussian measures are preserved.}

Let $\mu$ be sub-Gaussian and set $\bar \mu = \mathcal{P}_N(\mu)$, so that, by \eqref{transferop} and the definition of $\nu_u$ in \eqref{nuu}, $\bar \mu(dx)=\mathbb{E}_{\mu}[\nu_u(dx)].$

Fix an observable  $g :\mathbb{T}^N \to \mathbb{R}.$ Using the definition of $F_g(u)$ in \eqref{effgiu}, we have 
\begin{equation}\label{expectationdecomp}
\begin{aligned}
\mathbb{E}_{\bar \mu}\left[e^{\lambda(g-\int g d\bar \mu)}\right]
&= \mathbb{E}_{\mu}\left[e^{\lambda(F_g(u)-\int g d\bar \mu)}\mathbb{E}_{\nu_u}\left[e^{\lambda(g-F_g(u))}\right]\right]\\
&\le  \mathbb{E}_{\mu}\left[e^{\lambda\left(F_g(u)-\int gd\bar \mu\right)}\right]
\exp\left(\frac{\lambda^2}{8}\sum_{i=1}^N \operatorname{Lip}_i(g)^2\right)
\\
&
\le \exp\left(\frac{\lambda^2}{2}\sum_{i=1}^N \operatorname{Lip}_i(F_g)^2\right)\exp\left(\frac{\lambda^2}{8}\sum_{i=1}^N \operatorname{Lip}_i(g)^2\right),
\end{aligned}
\end{equation}

where, in the last line, we applied first to measure $\nu_u$ the concentration inequality for product measures in  \cite[Theorem 6.2]{boucheron2013concentration}, and then we used the fact that $\mu$ is sub-Gaussian, along with  the fact that 
$$
\int F_g(u)d\mu(u) = \int g\d\bar \mu.
$$

Combining \eqref{expectationdecomp} with \eqref{squarecontractionobservable}, we obtain 
\begin{align*}
\mathbb{E}_{\bar \mu}\left[e^{\lambda(g-\int g d\bar \mu)}\right]\le \exp\left[
\left(\frac{1}{8}
+ \frac{\kappa^2}{2}
\right)
\lambda^2 \sum_{i=1}^N \textnormal{Lip}_i(g)^2
\right],
\end{align*}
where $\kappa$ is as in \eqref{kappa}. Again using \eqref{kappa}, we get
$$
\frac{1}{8}
+ \frac{\kappa^2}{2} <
\frac{1}{8}+\frac{3}{8}
=
\frac{1}{2},
$$
and in particular we conclude that $\bar \mu$ is sub-Gaussian.

\medskip 

\textbf{Step 2: contraction of the $N$-dimensional transfer operator}

Let $\mu_1,\mu_2 \in \mathcal{M}_1(\mathbb{T}^N)$ and let $\mathcal{P}_N$ be the transfer operator in \eqref{transferop}.
By \eqref{transferop},\eqref{nuu}, and \eqref{effgiu}, we have that, for any  observable $g \colon \mathbb{T}^N \to \mathbb{R},$ 
\begin{align*}
\int g(x)(\mathcal{P}_N(\mu_1-\mu_2))(dx) = \int F_g(u) (\mu_1-\mu_2)(du) \le \operatorname{Lip}^N(F_g)W^1_N(\mu_1,\mu_2),
\end{align*}
from which \eqref{tensorizedcontracton} follows by \eqref{contractionobservable}. By Banach fixed point theorem, there exists unique fixed point $\rho_N \in \mathcal{M}_1(\mathbb{T}^N).$ The fact that $\rho_N$ is sub-Gaussian follows from the fact that, by step $1$, if $\mu$ is sub-Gaussian, so is $\mathcal{P}^k_N(\mu)$ for all $k\ge 0$, and, by \eqref{tensorizedcontracton}
$$
W^1_N(\mathcal{P}^k_N(\mu), \rho_N) \to 0,\qquad k\to \infty.
$$
\end{TheoremproofB}

\section{Proof of Theorem \ref{thm2}}\label{sec4}
In this section we prove Theorem \ref{thm2}. We first record an estimate that will be used often throughout the proof. 

Let $p$ be a probability density satisfying \eqref{doeblin} with $c >0$. As in \eqref{central_Deco}, we can write
\begin{align*}
p(x) = c+ (1-c)q(x),
\end{align*}
where $q:=\frac{p-c}{1-c}$ is a probability density. 
If $a,b \in \mathbb T$, using the above decomposition and the definition of $W^1_1$ in \eqref{wasserstainspace}, we get 
\begin{equation}\label{key1dwasserstein}
W^1_1(p(x-a)dx,p(x-b)dx)\le (1-c)d(a,b).
\end{equation}

\begin{TheoremproofC}

Let us fix $N\ge 2$, and the $N$-dimensional Markov chain satisfying Hypothesis \ref{A}, with one-particle map $f \in \operatorname{Lip}(\mathbb{T}), h \in \operatorname{Lip}(\mathbb{T}^2),$ density $p$ satisfying \eqref{doeblin} and $\kappa$ as in \eqref{kappa}. 

\medskip 

\textbf{Step 1: contraction estimate for the STO.}

We first prove the one-step contraction estimate for the self-consistent transfer operator in  item \ref{one}.  Let $\mathcal{P}$ denote the self-consistent transfer operator as in \eqref{STO}. Furthermore, let $\nu_1,\nu_2 \in \mathcal{M}_1(\mathbb T)$ and define, for $i=1,2$,  $\bar \nu_i :=\mathcal{P}(\nu_i)$. Let $p$ be as in \eqref{doeblin}.
By \eqref{STO}, if  $g \colon \mathbb{T} \to \mathbb{R}$ satisfies $\operatorname{Lip}(g) \le 1$, we can write 
\begin{equation}\label{decomposition}
\int g(x) d(\bar \nu_1-\bar \nu_2)(x) = I_1+I_2
\end{equation}
where 
\begin{equation}\label{i1}
I_1 := \int \left(\int g(x) p\left(x-f(u)-\int h(u,v)d\nu_1(v) \right)dx\right)d( \nu_1-\nu_2)(u)
\end{equation}
and 
\begin{equation}\label{i2}
I_2 := \int \left[\int g(x)\left( p\left(x-f(u)-\int h(u,v)d\nu_1(v) \right) -   p\left(x-f(u)-\int h(u,v)d\nu_2(v)\right)\right) dx\right]d\nu_2(u)
\end{equation}
To estimate $I_1$, note that
\begin{align*}
|I_1| &\le \operatorname{Lip}\left(u\mapsto \int g(x) p\left(x-f(u)-\int h(u,v)d\nu_1(v) \right)dx\right)W^1(\nu_1,\nu_2)\\
&\le \sup_{u \neq u'}\frac{\left|\int g(x) p\left(x-f(u)-\int h(u,v)d\nu_1(v) \right)dx-\int g(x) p\left(x-f(u')-\int h(u',v)d\nu_1(v) \right)dx
\right|}{d(u,u')}W^1(\nu_1,\nu_2)
\\
&\le (1-c)\left(\operatorname{Lip}(f)+\operatorname{Lip}_1(h)\right)W^1(\nu_1,\nu_2),
\end{align*}
where, between the second and the third lines, we used  \eqref{key1dwasserstein}.

To estimate $I_2$, note that
\begin{align*}
|I_2|
&\le
\int
W^1_1\left(
p\left(x-f(u)-\int h(u,v)d\nu_1(v)\right)dx,\,
p\left(x-f(u)-\int h(u,v)d\nu_2(v)\right)dx
\right)d\nu_2(u)
\\
&\le
(1-c)\int
\left|
\int h(u,v)d\nu_1(v)
-
\int h(u,v)d\nu_2(v)
\right|d\nu_2(u)
\\
&\le
(1-c)\operatorname{Lip}_2(h)W^1_1(\nu_1,\nu_2),
\end{align*}
where again, between the first and second lines, we used \eqref{key1dwasserstein}. Combining \eqref{decomposition}, \eqref{i1}, \eqref{i2}, the definition of $\kappa$ in \eqref{kappa} and the above estimates, we obtain
\[
W^1_1\bigl(\mathcal{P}(\nu_1),\mathcal{P}(\nu_2)\bigr)
\le
(1-c)\left(\operatorname{Lip}(f)+\operatorname{Lip}_1(h)+\operatorname{Lip}_2(h)\right)
W^1_1(\nu_1,\nu_2)
=
\kappa W^1_1(\nu_1,\nu_2),
\]
which proves item \ref{one}.

\medskip

\textbf{Step 2: invariance neighbourhood of $\rho^{\otimes N}$}

In this step we prove item \ref{two}. Let $\rho$ be the unique fixed point of $\mathcal{P}$, $\bar C$ as in \eqref{barC} and the set $B_{\bar C,N}(\rho)$ as defined in \eqref{BCN} with $C=\bar C$ and $q =\rho$. Let $\mu \in B_{\bar C,N}(\rho).$ By the contraction of $\mathcal{P}_N$ in \eqref{tensorizedcontracton}, we have 
\begin{equation}\label{invneigh}
W^1_N(\mathcal{P}_N(\mu),\rho^{\otimes N}) \le \kappa \bar C N^{-\frac{1}{2}} + W^1_N(\mathcal{P}_N(\rho^{\otimes N}),\rho^{\otimes N})
\end{equation}

We have
\begin{equation}\label{xi1}
\mathcal{P}_N(\rho^{\otimes N})(dx) = \int  \prod_{j=1}^N p_j (x_j,u) dx_j\rho^{\otimes N}(du),
\end{equation}
where  
\begin{equation}\label{pig}
p_j(x_j,u):= p\left(x_j-f(u_j) -\frac{1}{N}\sum_{k=1}^N h(u_j,u_k) \right),
\end{equation}
and, since $\mathcal{P}(\rho) = \rho$, with $\mathcal{P}$ as in \eqref{STO}, we get 
\begin{equation}\label{barxi}
\rho^{\otimes N}(dx) = \int  \prod_{j=1}^N \bar p(x_j,u_j) dx_j \rho^{\otimes N}(du),
\end{equation}
where 
$$
\bar p(x,u) = p\left(x-f(u) -\int  h(u,v) \rho(dv)\right).
$$
By \eqref{xi1} and \eqref{barxi}, given any observable $g,$ we can decompose
\begin{equation}\label{telescope}
\int g(x)(\mathcal{P}_N(\rho^{\otimes N})-\rho^{\otimes N})(dx) =  \sum_{l=1}^N \int  g(x)  (\xi_{l}-\xi_{l-1})(dx),
\end{equation}
where
\begin{align*}
\xi_l:= \int \prod_{i\le l}\left(p_i(x_i,u) dx_i \right)\prod_{i>l}(\bar p(x_i,u)dx_i)\rho^{\otimes N}_0(du),\qquad \forall l=0,\dots,N.
\end{align*}
Now observe that, for all  $l=1,\dots ,N$
\begin{align*}
&\int  g(x)  (\xi_{l}-\xi_{l-1})(dx) \\
&= \int\left(\int  g(x)\left[p_l(x_l,u)-\bar p(x_l,u_l)\right]dx_l\right) \prod_{i< l}(p_i(x_i,u)dx_i) \prod_{i>l}\bar (p(x_i,u)dx_i)\rho^{\otimes N}(du)\\
&\le \operatorname{Lip}_{l}(g) \int W^1_1(p_l(x_l,u)dx_l,\bar p(x_l,u_l)dx_l) \rho^{\otimes N}(du) \\
&\le \operatorname{Lip}_{l}(g)(1-c)\mathbb{E}_{\rho^{\otimes N}}\left[\left|\frac{1}{N}\sum_{k=1}^N h(u_l,u_k)-\int h(u_l,v)\rho(dv)\right|\right]
\\
&\le \operatorname{Lip}_{l}(g)(1-c)\left(\operatorname{Lip}_2(h)\right)N^{-\frac{1}{2}},
\end{align*}
where, between the third and the fourth lines, we used the decomposition in  \eqref{key1dwasserstein}, and in the last line we used the concentration inequalities for product measures in \cite[Theorem 6.2]{boucheron2013concentration}, after conditioning on $u_l.$ Summing the above over $l=1,\dots,N$ and using \eqref{telescope} we get 
\begin{equation}\label{firstw1}
W^1_N(\mathcal{P}_N(\rho^{\otimes N}),\rho^{\otimes N}) \le (1-c)\left(\operatorname{Lip}_2(h)\right)N^{-\frac{1}{2}},
\end{equation}
which, combined with \eqref{invneigh} and the definition of $\bar C$ in \eqref{barC}, proves the inclusion in \eqref{firstinclusion}. The last assertion follows from the $W^1_N$ global contraction of $\mathcal{P}_N$ in \eqref{tensorizedcontracton}.
\end{TheoremproofC}
\section{Propagation of chaos}\label{sec5}
In this section we prove Theorem \ref{thm3} and Corollary \ref{cor1}.
\begin{TheoremproofE}
Let $\mu \in \mathcal{M}_1(\mathbb T^N).$
We first prove \eqref{firstlasota}. Observe that 
\begin{equation}\label{chain}
\begin{aligned}
W_N^1\left(
\mathcal P_N^n(\mu),
\left(\mathcal P^n(\pi_1\mu)\right)^{\otimes N}
\right)
&\leq
W_N^1\left(
\mathcal P_N^n(\mu),
\mathcal P_N\left(
\left(\mathcal P^{n-1}(\pi_1\mu)\right)^{\otimes N}
\right)
\right)
\\
&+
W_N^1\left(
\mathcal P_N\left(
\left(\mathcal P^{n-1}(\pi_1\mu)\right)^{\otimes N}
\right),
\left(\mathcal P^n(\pi_1\mu)\right)^{\otimes N}
\right) \\
&\le \kappa W_N^1\left(
\mathcal P_N^{n-1}(\mu),
\left(\mathcal P^{n-1}(\pi_1\mu)\right)^{\otimes N}
\right) + (1-c)\operatorname{Lip}_2(h) N^{-\frac{1}{2}},
\end{aligned}
\end{equation}
where between the second and the last lines we used \eqref{tensorizedcontracton} and the same telescoping argument used to prove \eqref{firstw1}. Reiterating \eqref{chain}, we obtain \eqref{firstlasota}.

We now prove \eqref{secondlasota}. Let $p \in \operatorname{BV}(\mathbb{T}),$ then we have the following regularizing inequality.
\begin{equation}\label{Key}
\|p(\cdot-a)-p(\cdot-b)\|_{TV}\le \frac{\operatorname{Var}(p)}{2}d(a,b).
\end{equation}
For $1\le k\le N,$ observe that
\begin{equation}\label{TV}
\begin{aligned}
\left\|
\pi_{\{1,\dots,k\}}\mathcal{P}_N^n(\mu)
-
\left(\mathcal P^n(\pi_1\mu)\right)^{\otimes k}
\right\|_{TV}
&\le
\left\|
\pi_{\{1,\dots,k\}}\mathcal{P}_N\left(
\mathcal P_N^{n-1}(\mu)
-
\left(\mathcal P^{n-1}(\pi_1\mu)\right)^{\otimes N}
\right)
\right\|_{TV}
\\
&\quad+
\left\|
\mathcal{P}^{\otimes k}
\left(
\left(\mathcal P^{n-1}(\pi_1\mu)\right)^{\otimes k}
\right)
-
\pi_{\{1,\dots,k\}} \mathcal{P}_N
\left(
\left(\mathcal P^{n-1}(\pi_1\mu)\right)^{\otimes N}
\right)
\right\|_{TV}.
\end{aligned}
\end{equation}

For the first term, observe that, if $\phi$ is an observable with $|\phi|_{\infty} \le 1,$ then we have 
\begin{equation}\label{kappavar}
\begin{aligned}
&\int \phi(x)\prod_{i \in I} (p_i(x_i,u)dx_i)
d\left(
\mathcal P_N^{n-1}(\mu)
-
\left(\mathcal P^{n-1}(\pi_1\mu)\right)^{\otimes N}
\right)(du)
\\
&\qquad\le
\operatorname{Lip}^N\left(
u\mapsto
\int \phi(x)\prod_{i \in I} (p_i(x_i,u)dx_i)
\right)
W^1_N\left(
\mathcal P_N^{n-1}(\mu),
\left(\mathcal P^{n-1}(\pi_1\mu)\right)^{\otimes N}
\right),
\end{aligned}
\end{equation}
with $p_i$ as in \eqref{pig}. A telescoping  argument based on a decomposition analogous to the one used in \eqref{telescope}, along with \eqref{Key}, gives 
\begin{equation}\label{telescopingargument}
\operatorname{Lip}^N\left(
u  \mapsto 
\int \phi(x) \prod_{i \in I} (p_i(x_i,u)dx_i)
\right) 
\le 
k\operatorname{Var}(p)
\left(
\operatorname{Lip}(f)
+\operatorname{Lip}_1(h)
+\operatorname{Lip}_2(h)
\right)
:=L_k,
\end{equation}
so that, by \eqref{TV}, \eqref{kappavar} and \eqref{firstlasota},
\begin{align*}
&\left\|
\pi_{\{1,\dots,k\}} \mathcal{P}_N^n(\mu)
-
\left(\mathcal P^n(\pi_1\mu)\right)^{\otimes k}
\right\|_{TV}
\\
&\qquad\le
\frac{L_k}{2}
\left[
\kappa^{n-1}
W_N^1\left(
\mu,
(\pi_1\mu)^{\otimes N}
\right)
+
\bar C N^{-\frac12}
\right]
\\
&\qquad\quad+
\left\|
\mathcal{P}^{\otimes k}
\left(
\left(\mathcal P^{n-1}(\pi_1\mu)\right)^{\otimes k}
\right)
-
\pi_{\{1,\dots,k\}} \mathcal{P}_N
\left(
\left(\mathcal P^{n-1}(\pi_1\mu)\right)^{\otimes N}
\right)
\right\|_{TV}.
\end{align*}
A computation similar to the one made to establish \eqref{telescopingargument} gives 
\begin{align*}
&\left\|
\mathcal{P}^{\otimes k}
\left(
\left(\mathcal P^{n-1}(\pi_1\mu)\right)^{\otimes k}
\right)
-
\pi_I \mathcal{P}_N
\left(
\left(\mathcal P^{n-1}(\pi_1\mu)\right)^{\otimes N}
\right)
\right\|_{TV}
\\
&\qquad\le 
\frac{k\operatorname{Var}(p)}{2}
\operatorname{Lip}_2(h) N^{-\frac12}
:=
L'_k N^{-\frac12},
\end{align*}
from which we conclude that
\begin{align*}
&\left\|
\pi_I \mathcal{P}_N^n(\mu)
-
\left(\mathcal P^n(\pi_1\mu)\right)^{\otimes k}
\right\|_{TV}
\\
&\qquad\le
\frac{L_k}{2}\kappa^{n-1}
W_N^1\left(
\mu,
(\pi_1\mu)^{\otimes N}
\right)
+
\left(
\frac{L_k\bar C}{2}
+
L'_k
\right)
N^{-\frac12},
\end{align*}
and \eqref{secondlasota} follows.
\end{TheoremproofE}

\begin{TheoremproofA}
The proof of Corollary \ref{cor1} follows from the fact that if a family of exchangeable measures $\{\mu_N\}_{N\ge2}$ satisfies  \eqref{kcong}, then it satisfies also 
\begin{align*}
W_N^1\left(\mu_N,(\pi_1\mu_N)^{\otimes N}\right)\to0. \qquad  N \to \infty.
\end{align*}
This follows from the standard characterization of chaos in terms of convergence of the empirical measures \cite[Proposition 2.2]{Sznitman1991}, together with the representation of $W_N^1$ for exchangeable measures obtained by symmetrizing the test observables. Because of the above convergence, \eqref{firstlasota} implies \eqref{propagationw1n} and \eqref{secondlasota} implies \eqref{whatsithis}.
\end{TheoremproofA}
\bibliographystyle{plain} 
\bibliography{bibliog}
\end{document}